\documentclass[reqno]{amsart}
\usepackage{hyperref}
\usepackage{amssymb}

\begin{document}
\title[Self-similar blow-up solutions]
{Self-similar blow-up solutions of the KPZ equation}

\author[A. Gladkov ]
{Alexander Gladkov}  % in alphabetical order

\address{Alexander Gladkov \newline
Department of Mathematics and Mechanics \\
Belarusian State University \\
Nezavisimosti avenue,4,  220030, Minsk, Belarus}
\email{gladkoval@mail.ru}

%\thanks{Submitted November , 2014. Published }

\begin{abstract}
In this paper we consider self-similar blow-up solutions for the
generalized deterministic KPZ~equation $u_t = u_{xx} + \lambda
\vert u_x \vert ^q, \lambda > 0, q > 2.$ The asymptotic behavior
of self-similar solutions are studied.
\end{abstract}

\maketitle \numberwithin{equation}{section}
\newtheorem{theorem}{Theorem}[section]
\newtheorem{lemma}[theorem]{Lemma}
\newtheorem{remark}[theorem]{Remark}
\allowdisplaybreaks

\section{Introduction}\label{Sect.1}

We consider the generalized deterministic KPZ~equation
\begin{equation}\label{Eq.1.1}
\frac{\partial{u}}{\partial{t}} =
\frac{\partial^2{u}}{\partial{x}^2} + \lambda
\left|\frac{\partial{u}}{\partial{x}}\right|^q, \quad (x,t) \in
S_T := {\mathbb{R}} \times(0,T),
\end{equation}
where $\lambda,$ $q$ and $T $ are real constants, $\lambda > 0,$
$q>2,$ $T > 0.$ Equation (\ref{Eq.1.1}) was first considered in
the case $q=2$ by Kardar, Parisi and Zhang \cite{KPZ} in
connection with the study of the growth of surfaces. When $q=2$
(\ref{Eq.1.1}) has since been referred to as the deterministic
KPZ~equation. For $q \neq 2$ it also called the generalized
deterministic KPZ~equation or Krug-Spohn equation because it was
introduced in \cite{KS}. We refer to review article \cite{C} for
references and a detailed historical account of the KPZ equation.

The existence and uniqueness of a classical solution of the Cauchy
problem for (\ref{Eq.1.1}) with $q=1$ and initial function $u_0
\in C_0^3({\mathbb{R}}^n)$ was proven in \cite{Benar}. This result
was extended to $u_0 \in C^2({\mathbb{R}}^n) \cap
W^{2,\infty}({\mathbb{R}}^n)$ and $q \geq 1$ in \cite{AmoB} and to
$u_0 \in C({\mathbb{R}}^n) \cap L^{\infty}({\mathbb{R}}^n)$ and $q
\geq 0$ in \cite{GGK}. Several papers \cite{BL},
 \cite{Benac1}, \cite{BGL}, \cite{BSW2}, \cite{BSW} were devoted to the investigation of the Cauchy
problem for irregular initial data, namely for $u_0 \in
L^p({\mathbb{R}}^n),$ $1\leq p<\infty$ or for bounded measures.
The existence and uniqueness of a solution to the Cauchy problem
with unbounded initial datum are proved in \cite{GlGK}. To confirm
the optimality of obtained existence conditions the authors of
\cite{GlGK} analyze the asymptotic behavior of self-similar
blow-up solutions of (\ref{Eq.1.1}) for $q < 2.$

In this paper for $q > 2$ we investigate the asymptotic behavior
of self-similar blow-up solutions of (\ref{Eq.1.1}) of the
following form
\begin{equation}\label{Eq.4.1}
u(x,t)=(T_0-t)^\alpha f(\xi), \quad \mbox {where } \xi=\vert x \vert
(T_0-t)^\beta.
\end{equation}
Here $T_0$ is any positive constant.
After substitution (\ref{Eq.4.1}) into (\ref{Eq.1.1}) we find that
\begin{equation}\label{Eq.4.2}
\alpha = \frac{q-2}{2(q-1)}, \quad
\beta = -\frac{1}{2}
\end{equation}
and $f$ should satisfy the following equation
\begin{equation}\label{Eq.4.3}
{\mathcal L} f:= f^{\prime\prime} + \lambda \vert f^\prime \vert^q
-\frac{1}{2} \xi f^\prime + \frac{q-2}{2(q-1)} f = 0.
\end{equation}

We shall add to equation (\ref{Eq.4.3}) the following initial data
\begin{equation}\label{Eq.4.16}
f(0) = -1, \quad f^\prime (0)=0.
\end{equation}
A more general initial condition $f(0) = f_0 < 0$ may be
transformed by introducing a new function $f/\vert f_0 \vert$.  \\
Put
\begin{equation}\label{Eq.4.7}
C=\left[\frac{1}{\lambda(q-1)} \left( \frac{q-1}{q} \right)^q
\right]^ {1/(q-1)}.
\end{equation}

Our aim is to investigate the asymptotic behavior of solution of
(\ref{Eq.4.3}), (\ref{Eq.4.16}) for $q>2.$ The main result of the
paper is following.

\begin{theorem}\label{Thm.1}
Let $q>2$ and $f(\xi)$ be a solution of  problem (\ref{Eq.4.3}),
(\ref{Eq.4.16}). Then
\begin{equation}\label{Eq.4.40}
\lim_{\xi\to\infty} \frac{f(\xi)}{\xi^{q/(q-1)}} = C,
\end{equation}
where $C$ is defined in (\ref{Eq.4.7}).
\end{theorem}

The behavior of self-similar solutions of type $u(x,t)= t^\alpha
g(xt^\beta)$ for (\ref{Eq.1.1}) has been analyzed in \cite{GuK}.

\section{The proof of Theorem~\ref{Thm.1}}\label{Sect.2}

\noindent We start with simple result which is used later on.
\begin{lemma}\label{Lem.7}
Let $f(\xi)$ be a solution of  problem {\rm(\ref{Eq.4.3})},
{\rm(\ref{Eq.4.16})} defined for $\xi \in [0,\bar{\xi})$. Then
\begin{equation}\label{Eq.4.5}
f^\prime (\xi) > 0 \quad  {\rm and} \quad f^{\prime\prime} (\xi) >
0 \,\, \mbox {for }  \, \xi \in (0,\bar{\xi}).
\end{equation}
\end{lemma}
\begin{proof} It is obvious, $f^{\prime\prime} (0) =
(q-2)/[2(q-1)] > 0.$ Therefore by continuity $f^{\prime\prime}
(\xi) > 0$ and $f^\prime (\xi) > 0$ in the right-hand side of some
neighborhood of point $\xi = 0$. We suppose that there exists
$\xi_0$ such that $0<\xi_0<\bar{\xi}, \,$ $f^{\prime\prime} (\xi)
> 0$ for $0 < \xi < \xi_0$ and $f^{\prime\prime} (\xi_0) = 0.$
Then $f^{\prime} (\xi)
> 0$ for $0 < \xi \leq \xi_0$ and $f^{\prime\prime\prime} (\xi_0) \leq 0.$
Equation (\ref{Eq.4.3}) at $\xi_0$ subsequently leads to a
contradiction.
\end{proof}

Now we shall obtain the upper bound for $f^\prime (\xi)$.
\begin{lemma}\label{Lem.10}
There exists $\xi_0 > 0$ such that
\begin{equation}\label{Eq.4.17}
f^\prime(\xi) < \left\{ \frac{\xi}{2\lambda}
\right\}^\frac{1}{q-1} \, \mbox {for } \, \xi \geq \xi_0.
\end{equation}
\end{lemma}
\begin{proof}

Lemma~\ref{Lem.7} implies that $f(\xi) \to \infty$ as $\xi \to
\bar{\xi},$ and, there exists unique point $\xi_0 \in
(0,\bar{\xi})$ such that $f (\xi) <0$ for  $\xi \in (0,\xi_0)$ and
$f (\xi)> 0$ for   $\xi \in (\xi_0,\bar{\xi}).$ Substituting
$f^{\prime\prime} >0$ and $f \ge 0$ in (\ref{Eq.4.3}) yields
$f^{\prime} < \{\xi/(2\lambda)\}^{1/(q-1)}$ for  $\xi \in
[\xi_0,\bar{\xi}).$ Thus, $\bar{\xi}= \infty$ and (\ref{Eq.4.17})
holds.
\end{proof}

Changing variables in (\ref{Eq.4.3})
\begin{equation}\label{Eq.4.20}
f^\prime(\xi)=\xi^{1/(q-1)} g(t), \quad \xi = \exp t,
\end{equation}
we get new equation
\begin{equation}\label{Eq.4.21}
g^{\prime\prime} + \frac{3-q}{q-1}g^\prime - \frac{q-2}{(q-1)^2}g =
\left\{ \frac{1}{2}g^\prime - \lambda (g^q)^\prime + \frac{1}{q-1}g -
\frac{\lambda q}{q-1} g^q \right\} \exp (2t).
\end{equation}
By (\ref{Eq.4.5}), (\ref{Eq.4.17}) and (\ref{Eq.4.20}) there hold
\begin{equation}\label{Eq.4.22}
 g(t) > 0 \quad \mbox {for any } t \in \mathbb{R},
\end{equation}
\begin{equation}\label{Eq.4.23}
g(t) < \left\{ \frac{1}{2\lambda} \right\}^\frac{1}{q-1} \quad
\mbox {and} \quad g^\prime (t) > -\frac{g}{q-1}
\end{equation}
for sufficiently large values of $t$. Put
\begin{equation}\label{Eq.4.24}
C_0 =  \left\{ \frac{1}{\lambda q} \right\}^\frac{1}{q-1} , \quad
C_1 =  \left\{ \frac{1}{2\lambda q} \right\}^\frac{1}{q-1}.
\end{equation}
It is obvious, $C_0 > C_1$. Now we shall establish the asymptotic
behavior of $g(t)$ as $t \rightarrow +\infty$.
\begin{lemma}\label{Lem.11}
Assume that $g(t)$ is defined in {\rm (\ref{Eq.4.20})}. Then
\begin{equation}\label{Eq.4.25}
\lim_{t\rightarrow +\infty} g(t) = C_0.
\end{equation}
\end{lemma}
\begin{proof}
From a careful inspection of equation (\ref{Eq.4.21}) we conclude
that a local maximum of $g(t)$ can happen  only when $g(t) > C_0$.

At first we suppose that $g(t)$ does not tend to $C_0$ as
$t\rightarrow +\infty$ and $g(t)$ is monotonic solution of
(\ref{Eq.4.21}). Then there exists $\bar C \neq C_0$ such that
$\lim_{t \to \infty} g(t) = \bar C.$ It is obvious, for any
$\varepsilon
>0$ and some $A>0$ there exists a sequence $\{ t_k \}$ with the
properties:
$$
\lim_{k \to \infty}t_k = +\infty, \, \vert g^{\prime\prime}(t_k)
\vert \leq A, \,  \vert g^\prime (t_k) \vert \leq \varepsilon.
$$
Passing to the limit in (\ref{Eq.4.21}) as $t=t_k \to +\infty$ and
choosing $\varepsilon$ in a suitable way we get that left-hand
side is bounded, while right-hand side tends to infinity if $\bar
C \neq 0.$ Let $\bar C = 0.$ Using (\ref{Eq.4.23}) we conclude
from (\ref{Eq.4.21}) that
\begin{equation}\label{Eq.4.26}
g^{\prime\prime} + \frac{3-q}{q-1}g^\prime \ge \frac{g}{3(q-1)}
\exp (2t)
\end{equation}
for large values of $t.$ Then for large values of $k$
(\ref{Eq.4.26}) implies
\begin{equation}\label{Eq.4.27}
 g (t_k) \leq \alpha \exp (-2t_k),
\end{equation}
where positive constant $\alpha$ does not depend on $k.$ Setting $
\xi_k = \exp t_k,$ from (\ref{Eq.4.20}), (\ref{Eq.4.27}) we get
$$
f^\prime(\xi_k) \leq \alpha \xi_k^{(3-2q)/(q-1)},
$$
that contradicts (\ref{Eq.4.5}).

Now until the end of the proof we assume that $g(t)$ is not
monotonic solution of (\ref{Eq.4.21}). Suppose that $\lim\inf_{t
\to \infty} g(t)<C_0.$ Then there exist positive unbounded
increasing sequences $\{s_k\}$ and $\{t_k\}\,$ such that $t_k >
s_k,$
\begin{equation}\label{Eq.4.28}
g^\prime (t) \leq 0, \,\,\, t \in [s_k, t_k],
\end{equation}
and $g(s_k)=C_0,\,$ $g(t_k)=C_\star$ for $k \in \mathbb {N},$
where $C_1 < C_\star < C_0.$ Then
\begin{equation}\label{Eq.4.29}
\frac{1}{2} g^\prime -\lambda (g^q)^\prime = -\lambda q (g^{q-1}
-C_1^{q-1)}) g^\prime \geq-\lambda q (C_\star^{q-1} -C_1^{q-1)})
g^\prime \geq 0 \,\, \text {\rm on} \,\, [s_k, t_k].
\end{equation}
So, (\ref{Eq.4.21}) and (\ref{Eq.4.29}) imply that
$$
 g^{\prime\prime} (t) + \frac{3-q}{q-1}  g^{\prime} (t) \geq -\lambda q (C_\star^{q-1}
-C_1^{q-1)}) g^\prime (t) \exp (2 s_k), \,\, t \in [s_k, t_k].
$$
Hence, integrating with respect to $t$ from $s_k$ to $t_k,$ we get
$$
 \left\{ g^{\prime} (t) + \frac{3-q}{q-1}  g (t) \right\}
 \Bigr|^{t_k}_{s_k} \geq \lambda q (C_\star^{q-1}
-C_1^{q-1)}) (C_0 - C_\star) \exp (2 s_k).
$$
This leads to a contradiction, since (\ref{Eq.4.22}),
(\ref{Eq.4.23}) and (\ref{Eq.4.28}) imply that the left-hand side
of last inequality is bounded, while the right-hand side becomes
unbounded as $k \to \infty.$

Let us prove that $\lim\inf_{t \to \infty} g(t)=C_0.$ Indeed,
otherwise there exist $\varepsilon > 0$ and sequence $\{ \tau_k
\}$ of points of local minima $g(t)$ with the properties: $\tau_k
\to +\infty$ as $k \to +\infty$ and
\begin{equation}\label{Eq.4.33}
g(\tau_k) \geq C_0 + \varepsilon.
\end{equation}
Passing in (\ref{Eq.4.21}) to the limit as $t=\tau_k \to +\infty$
we get a contradiction.

To end the proof we show that $\lim\sup_{t \to \infty} g(t)=C_0.$
Otherwise $\lim\sup_{t \to \infty} g(t)>C_0.$ Then there exist
unbounded increasing sequences $\{s_k\}$ and $\{t_k\}\,$ such that
$t_k > s_k > 2,$
\begin{equation}\label{Eq.4.34}
g^\prime (s_k) =0, \, g^\prime (t_k) =0, \, g^\prime (t) \geq 0 \,
\, \mbox {for } \, t \in [s_k, t_k], \, g(t_k) > C_0 + \delta, \,
\vert g(s_k)-C_0 \vert < \varepsilon,
\end{equation}
where $k \in \mathbb {N}, \,$ $\delta >0, \,$
\begin{equation}\label{Eq.4.35}
\varepsilon = \min \left\{ \delta/2, \frac{q-1}{4C_0}\delta^2,
\left[ 1-\left( 7/8 \right)^\frac{1}{q-1} \right]C_0\right\}.
\end{equation}
Without loss of a generality we can suppose
\begin{equation}\label{Eq.4.36}
C_0 - \varepsilon < g(s_k) < C_0
\end{equation}
or
\begin{equation}\label{Eq.4.37}
C_0 \leq g(s_k) < C_0 + \varepsilon.
\end{equation}
Let (\ref{Eq.4.36}) be valid. If (\ref{Eq.4.37}) is realized the
arguments are similar and simpler. Denote by $\bar t_k \in
(s_k,t_k)$ points such that
\begin{equation}\label{Eq.4.38}
g(\bar t_k)=C_0.
\end{equation}
Applying H\"older's inequality we derive
\[
\int_{\bar t_k}^{t_k} g^\prime \, d\tau \leq \left( \int_{\bar
t_k}^{t_k} \left( g^\prime \right) ^2 \exp (2\tau) \, d\tau
\right)^{1/2} \left( \int_{\bar t_k}^{t_k} \exp (-2\tau) \, d\tau
\right)^{1/2}
\]
and therefore
\begin{equation}\label{Eq.4.39}
\int_{\bar t_k}^{t_k} \left( g^\prime \right) ^2 \exp (2\tau) \,
d\tau \geq 2 \delta^2 \exp(2\bar t_k).
\end{equation}
We multiply (\ref{Eq.4.21}) by $g^\prime (t)$ and integrate after
over $[s_k, t_k].$ Using (\ref{Eq.4.24}),
(\ref{Eq.4.34})--(\ref{Eq.4.36}), (\ref{Eq.4.38}) and
(\ref{Eq.4.39}) we obtain
\begin{eqnarray*}
-\frac{q-2}{2(q-1)} g^2 (t_k) & \leq & \frac{q-3}{q-1}
\int_{s_k}^{t_k} ( g^\prime )^2 \, d\tau + \int_{s_k}^{t_k} (
g^\prime )^2 \left[ \frac{1}{2} - \lambda q g^{q-1} \right]
\exp (2\tau) \, d\tau \\
&+& \frac{ \exp (2\bar t_k) }{q-1}  \int_{s_k}^{\bar t_k} \left[
\frac{1}{2}( g^2)^\prime - \frac{\lambda q}{q+1} (g^{q+1})^\prime
\right] \, d\tau \\
&\leq& - \frac{1}{4} \int_{\bar t_k}^{t_k} ( g^\prime )^2 \exp
(2\tau) \, d\tau + \frac{\exp (2\bar t_k)}{q-1}  \left(
\frac{g^2}{2} - \frac{\lambda q}{q+1} g^{q+1} \right)
\Bigr|_{t=s_k}^{t=\bar t_k} \\
& \leq & \left[ - \frac{\delta^2}{2} + \frac{\varepsilon C_0}{q-1}
\right] \exp (2\bar t_k) \leq - \frac{\delta^2}{4} \exp (2\bar
t_k) .
\end{eqnarray*}
Passing to the limit as $k \to \infty$ we get a contradiction with
(\ref{Eq.4.23}).
\end{proof}

From Lemma~\ref{Lem.11} and the definition of $g(t)$ it follows
(\ref{Eq.4.40}). Theorem~\ref{Thm.1} is proved. \qed

\begin{remark}
{\rm Theorem~\ref{Thm.1} demonstrates the optimality of
Theorem~2.3 in \cite{GlGK}. The arguments are the same as in
 Remark 4.6 of that paper.}
\end{remark}

Our next result shows that equation (\ref{Eq.4.3}) with initial
data
\begin{equation}\label{Eq.4.41}
f(0) = f_0 > 0, \quad f^\prime (0) = 0
\end{equation}
has no global solution.

\begin{theorem}\label{Thm.7}
Let $q>2$ and $f(\xi)$ be a solution of  problem (\ref{Eq.4.3}),
(\ref{Eq.4.41}). Then there exists $\xi_\star$ such that
$0<\xi_\star<+\infty$ and $f(\xi) \to -\infty$ as $\xi \uparrow
\xi_\star$  .
\end{theorem}
\begin{proof}
Suppose that problem (\ref{Eq.4.3}), (\ref{Eq.4.41}) has
infinitely extendable to the right solution $f(\xi).$ Using
arguments of Lemma~\ref{Lem.7} we show that $f^\prime (\xi) < 0
\quad \mbox {and } \quad f^{\prime\prime} (\xi) < 0 \quad \mbox
{for any } \xi > 0.$ From (\ref{Eq.4.3}) we obtain
\begin{equation}\label{Eq.0.2}
f^{\prime\prime\prime}(\xi) < -\lambda \left( \vert f^\prime(\xi)
\vert^q \right)^\prime.
\end{equation}
After integration (\ref{Eq.0.2}) over $[0,\xi]$ we conclude that
\begin{equation}\label{Eq.0.3}
f^{\prime\prime}(\xi) < -\lambda \vert f^{\prime}(\xi) \vert^q.
\end{equation}
Integrating (\ref{Eq.0.3}) over $[\xi_1,\xi]$ $(0 < \xi_1 < \xi)$
we infer
\[
\frac{1}{(q-1)\vert f^{\prime}(\xi_1) \vert^{q-1}} > \lambda (\xi
- \xi_1).
\]
Passing to the limit as $\xi \to \infty$  we obtain a
contradiction which proves the theorem. \end{proof}

\end{document}